\documentclass[graybox]{svmult}

\usepackage{mathptmx}       
\usepackage{helvet}         
\usepackage{courier}        
\usepackage{type1cm}        
\usepackage{makeidx}         
\usepackage{graphicx}        
\usepackage{multicol}        
\usepackage[bottom]{footmisc}

\makeindex             

\newcommand{\bX}{\mathbf{X}}
\newcommand{\bx}{\mathbf{x}}

\usepackage{amsmath} 
\usepackage{amssymb}
\usepackage{bm}
\usepackage{amsfonts}
\usepackage{enumerate}
\usepackage{natbib}
\bibpunct{(}{)}{;}{a}{,}{,}

\begin{document}

\title*{High-Dimensional $p$-Norms}
\author{G\'erard Biau and David M. Mason}

\institute{G\'erard Biau \at Universit\'e Pierre et Marie Curie, Ecole Normale Sup{\'e}rieure \& Institut universitaire de France, \email{gerard.biau@upmc.fr}
\and David M. Mason \at University of Delaware, \email{davidm@udel.edu}}

\maketitle

\abstract{
Let $\bX=(X_1, \hdots, X_d)$ be a $\mathbb R^d$-valued random vector with i.i.d.~components, 
and let $\Vert\bX\Vert_p= ( \sum_{j=1}^d|X_j|^p)^{1/p}$ be its $p$-norm, for $p>0$. The impact of letting $d$
go to infinity on $\Vert\bX\Vert_p$ has surprising consequences, which may dramatically affect high-dimensional 
data processing. This effect is usually referred to as the {\it distance concentration phenomenon} 
in the computational learning literature. Despite a growing interest in this important question, previous work has essentially characterized the problem in terms of numerical experiments and incomplete mathematical statements. In the present paper, we solidify some of the arguments which previously appeared in the literature and offer new insights into the phenomenon.
}

\section{Introduction}
In what follows, for $\bx=(x_1, \hdots, x_d)$ a vector of $\mathbb R^d$ and $0<p<\infty$, we set
\begin{equation}
\label{un}
\Vert\bx\Vert_p= \left( \sum_{j=1}^d|x_j|^p\right)^{1/p}.
\end{equation}
It is recalled that for $p\geq 1$, $\Vert.\Vert_p$ is a norm on $\mathbb R^d$ (the $L^p$-norm) but for $0<p<1$, the triangle inequality does not hold
and $\Vert.\Vert_p$ is sometimes called a prenorm. In the sequel, we take the liberty to call $p$-norm a norm or prenorm
of the form (\ref{un}), with $p>0$. 

Now, let $\bX=(X_1, \hdots, X_d)$ be a $\mathbb R^d$-valued random vector with i.i.d.~components. 
The study of the probabilistic properties of $\Vert\bX\Vert_p$ as the dimension $d$ tends to infinity has recently witnessed
an important research effort in the computational learning community \citep[see, e.g.,][for a review]{FrWeVe07}.
This activity is easily explained by the central role played by the quantity $\Vert\bX\Vert_p$ in the analysis of nearest neighbor search algorithms, which are currently widely used in data management and database mining. Indeed, finding the closest matching object in an $L^p$-sense is of significant importance for numerous applications, including pattern recognition, multimedia content retrieving (images, videos, etc.), data mining, fraud detection and {\sc dna} sequence analysis, just to name a few. Most of these real applications involve very high-dimensional data (for example, pictures taken by a standard camera consist of several million pixels) and the curse of dimensionality (when $d \to \infty$) tends to be a major obstacle in the development of nearest neighbor-based techniques. 

The effect on $\Vert\bX\Vert_p$ of letting $d$ go large is usually referred to as the {\it distance concentration phenomenon} in the computational learning literature. It is in fact a quite vague term that encompasses several interpretations. For example, it has been observed by several authors \citep[e.g.,][]{FrWeVe07} that, under appropriate moment assumptions, the so-called {\it relative standard deviation} ${\sqrt{\text{Var} \Vert\bX\Vert_p}}/{\mathbb E \Vert\bX\Vert_p}$
tends to zero as $d$ tends to infinity. Consequently, by Chebyshev's inequality (this will be rigorously established in Section 2), for all $\varepsilon>0$,
$$\mathbb P \left \{ \left |\frac{\Vert\bX\Vert_p}{\mathbb E \Vert\bX\Vert_p}-1\right|\geq\varepsilon\right\}\to 0, \text{ as } d \to \infty.$$
This simple result reveals that the relative error made as considering $\mathbb E \Vert\bX\Vert_p$ instead of the random value $\Vert\bX\Vert_p$ becomes asymptotically negligible. Therefore, high-dimensional vectors $\bX$ appear to be distributed on a sphere of radius $\mathbb {E}\Vert\bX\Vert_p$. 

The distance concentration phenomenon is also often expressed by considering an i.i.d.~$\bX$ sample $\bX_1, \hdots, \bX_n$ and observing that, under certain conditions, the {\it relative contrast}
$$\frac{\max_{1 \leq i \leq n} \Vert\bX_i\Vert_p-\min_{1\leq i \leq n} \Vert\bX_i\Vert_p}{\min_{1\leq i \leq n} \Vert\bX_i\Vert_p}$$
vanishes in probability as $d$ tends to infinity, whereas the {\it contrast} 
$$\max_{1 \leq i \leq n} \Vert\bX_i\Vert_p-\min_{1\leq i \leq n} \Vert\bX_i\Vert_p$$ 
behaves in expectation as $d^{1/p-1/2}$ \citep[][]{BeGoRaSh99,HiAgKe00,AgHiKe01,Ka12}. Thus, assuming that the query point is located at the origin, the ratio between the largest and smallest $p$-distances from the sample to the query point becomes negligible as the dimension increases, and all points seem to be located at approximately the same distance from the origin. This phenomenon may dramatically affect high-dimensional data processing, analysis, retrieval and indexing, insofar these procedures rely on some notion of $p$-norm. Accordingly, serious questions are raised as to the validity of many nearest neighbor search heuristics in high dimension, a problem that can be further exacerbated by techniques that find approximate neighbors in order to improve algorithmic performance \citep[][]{BeGoRaSh99}.

Even if people have now a better understanding of the distance concentration phenomenon and its practical implications, it is however our belief that there is still a serious need to solidify its mathematical background. Indeed, previous work has essentially characterized the problem in terms of numerical experiments and (often) incomplete probabilistic statements, with missing assumptions and (sometimes) defective proofs. Thus, our objective in the present paper is to solidify some of the statements which previously appeared in the computational learning literature. We start in Section 2 by offering a thorough analysis of the behavior of the $p$-norm $\Vert\bX\Vert_p$ (as a function of $p$ and the properties of the distribution of $\bX$) as $d\to \infty$. Section 3 is devoted to the investigation of some new asymptotic properties of the {\it contrast} $\max_{1 \leq i \leq n} \Vert\bX_i\Vert_p-\min_{1\leq i \leq n} \Vert\bX_i\Vert_p$, both as $d\to \infty$ and $n\to \infty$. For the sake of clarity, most technical proofs are gathered in Section 4.
\section{Asymptotic behavior of $p$-norms}
\subsection{Consistency}
Throughout the document, the notation $\stackrel{\mathbb P}{\to}$ and $\stackrel{\mathcal D%
}{\to}$ stand for convergence in probability and in distribution, respectively. The notation $u_n=\text{o}(v_n)$  and $u_n=\text{O}(v_n)$ mean, respectively, that $u_n/v_n\to 0$ and $u_n \leq C v_n$ for some constant $C$, as $n\to \infty$. The symbols
$\text{o}_{\mathbb{P}}(v_{n})$ and $\text{O}_{\mathbb{P}}(v_{n})$ denote, respectively, a sequence of random variables $\{  Y_{n}\}_{n \geq 1}$ such that $Y_{n}/v_{n}$ $\overset{\mathbb{P}%
}{\to}0$ and $Y_{n}/v_{n}$ is
bounded in probability, as $n\to \infty$.

We start this section with a general proposition that plays a key role in the analysis.
\begin{proposition}
\label{PR1}
Let $\{U_{d}\} _{d\geq1}$ be a sequence of random variables such that $U_{d}\stackrel{\mathbb P}{\to}a$, and let $\varphi$
be a real-valued measurable function which is continuous at $a$.  Assume that:
\begin{enumerate}[$(i)$]
\item $\varphi$ is bounded on $[-M,M]$ for some $M> \vert a \vert$;
\item $\mathbb E \vert \varphi (  U_{d})\vert <\infty$ for all $d\geq1$.
\end{enumerate}
Then, as $d\to \infty$,
\begin{equation*}
\mathbb E\varphi (  U_{d})  \to\varphi (  a)
\label{concl}%
\end{equation*}
if and only if
\begin{equation}
\mathbb E\left(  \varphi\left( U_{d} \right)  \mathbf 1\left\{
\vert U_{d}\vert >M\right\}  \right)  \to0.
\label{concl2}%
\end{equation}
\end{proposition}
\begin{proof}
The proof is easy. Condition $(i)$ and continuity of $\varphi$ at
$a$ allow us to apply the bounded convergence theorem to get
\begin{equation*}
\mathbb E\left(  \varphi (  U_{d})  \mathbf 1\left\{  \vert U_{d}\vert
\leq M\right\}  \right)  \to\varphi (a).
\label{c1}%
\end{equation*}
Since
\[
\mathbb E\varphi (  U_{d})  =\mathbb E\left(  \varphi (  U_{d})  \mathbf 1\left\{
\vert U_{d}\vert \leq M\right\}  \right)  +\mathbb E\left(  \varphi (
U_{d}) \mathbf 1\left\{  \vert U_{d}\vert >M\right\}  \right),%
\]
the rest of the proof is obvious. 
\qed
\end{proof}

We shall now specialize the result of Proposition \ref{PR1} to the case when
\begin{equation*}
U_{d}=d^{-1}\sum_{j=1}^{d}Y_{j}:=\overline{Y}_{d},
\label{sum}%
\end{equation*}
where $\{Y_{j}\}_{j \geq 1}$ is a sequence of i.i.d.~$Y$ random variables
with finite mean $\mu$. In this case, by the strong law of large numbers,
$U_{d}\to\mu$ almost surely. The following lemma
gives two sufficient conditions for (\ref{concl2}) to hold when $U_d=\overline{Y}_{d}$.
\begin{lemma}
\label{L1}
let $\varphi$ be a real-valued measurable function. Assume that one of the following two conditions is satisfied:
\begin{enumerate}[{\bf {Condition} 1}]
\item The function $\vert \varphi\vert$ is convex on $\mathbb{R}$ and $\mathbb E\vert \varphi (
Y) \vert <\infty$.
\item[]
\item For some $s>1$,%
\begin{equation*}
\limsup_{d\to\infty}\mathbb E\left\vert \varphi (  \overline{Y}%
_{d})  \right\vert ^{s}<\infty.
\label{s}%
\end{equation*}
\end{enumerate}
Then (\ref{concl2}) is satisfied for the sequence $\{\overline{Y}_{d}\}_{d \geq 1}$ with $a=\mu$ and $M> \vert \mu\vert$. 
\end{lemma}
\begin{proof}
Suppose that {\bf Condition 1} is satisfied. Then
note that by the convexity assumption%
\begin{align*}
\mathbb E\left(  \left\vert \varphi (  \overline{Y}_{d})  \right\vert
\mathbf 1\left\{  \vert \overline{Y}_{d}\vert >M\right\}  \right) & \leq
d^{-1}\sum_{j=1}^{d}\mathbb E\left(  \left\vert \varphi (  Y_{j})
\right\vert \mathbf 1\left\{  \vert \overline{Y}_{d}\vert >M\right\}
\right)\\
&=\mathbb E\left(  \left\vert \varphi (  Y)  \right\vert \mathbf 1\left\{
\vert \overline{Y}_{d}\vert >M\right\}  \right)  .
\end{align*}
Since $M>\vert \mu\vert $, we conclude that with
probability one, $\vert \varphi (  Y) \vert
\mathbf 1 \{  \vert \overline{Y}_{d}\vert >M\}  \to0$.
Also $\vert \varphi (  Y) \vert \mathbf 1 \{  \vert
\overline{Y}_{d}\vert >M\}  \leq \vert \varphi (
Y)  \vert $. Therefore, by the dominated convergence theorem,
(\ref{concl2}) holds.

Next, notice by H\"{o}lder's inequality with
$1/r=1-1/s$ that
\[
\mathbb E\left(  \left\vert \varphi (  \overline{Y}_{d})  \right\vert \mathbf 1\left\{  \vert
\overline{Y}_{d}\vert >M\right\}  \right)  \leq\left(  \mathbb E\left\vert
\varphi (  \overline{Y}_{d})  \right\vert ^{s}\right)  ^{1/s}\left(
\mathbb P\left\{  \vert \overline{Y}_{d}\vert >M\right\}  \right)  ^{1/r}.
\]
Since $\mathbb P\{  \vert \overline{Y}_{d}\vert >M\}
\to0$, (\ref{concl2}) immediately follows from {\bf Condition 2}.
\qed
\end{proof}

Let us now return to the distance concentration problem, which has been discussed in the introduction. Recall that we denote
by $\bX=(X_1, \hdots, X_d)$ a $\mathbb R^d$-valued random vector with i.i.d.~X components. 
Whenever for $p>0$
$\mathbb E\vert X\vert ^{p}<\infty$, we set $\mu_{p}=\mathbb E \vert X\vert
^{p}$. Also when $\text{Var}\vert X\vert ^{p}  <\infty
$, we shall write $\sigma_{p}^{2}=\text{Var}  \vert X\vert
^{p}$. Proposition \ref{PR1} and Lemma \ref{L1} yield the following corollary:
\begin{corollary}
\label{C1}
Fix $p>0$ and $r>0$.
\begin{enumerate}[$(i)$]
\item Whenever $r/p<1$ and $\mathbb E\vert X\vert ^{p}<\infty$,
\begin{equation*}
\frac{\mathbb E\Vert \bX \Vert_p^r}{d^{r/p}}\to\mu_{p}^{r/p},\text{ as }d\to\infty,
\label{d1}%
\end{equation*}
whereas if $\mathbb E\vert X\vert ^{p}=\infty$, then
\begin{equation*}
\lim_{d \to \infty}\frac{\mathbb E\Vert \bX \Vert_p^r}{d^{r/p}}=\infty.
\label{d11}%
\end{equation*}
\item Whenever $r/p\geq1$ and $\mathbb E\vert X\vert ^{r}<\infty$,
\begin{equation*}
\frac{\mathbb E\Vert \bX \Vert_p^r}{d^{r/p}}\to\mu_{p}^{r/p},\text{ as }d\to\infty,
\label{d2}%
\end{equation*}
whereas if $\mathbb E\vert X\vert ^{r}=\infty$, then, for all $d\geq1$,
\begin{equation*}
\frac{\mathbb E\Vert \bX \Vert_p^r}{d^{r/p}}=\infty.
\end{equation*}
\end{enumerate}
\end{corollary}
\begin{proof}
We shall apply Proposition \ref{PR1} and Lemma \ref{L1} to $Y=\vert
X\vert ^{p}$, $Y_{j}=\vert X_{j}\vert ^{p}$, $j\geq1$, and
$\varphi (  u)  =\vert u\vert ^{r/p}$.
\paragraph{Proof of $(i)$} For the first part of $(i)$, notice that with
$s=p/r>1$
\[
\mathbb E\left \vert \varphi\left(  \frac{\sum_{j=1}^{d}\vert X_{j}\vert ^{p}}%
{d}\right) \right\vert  ^{s}=\frac{\sum_{j=1}^{d}\mathbb E \vert X_{j}\vert ^{p}}%
{d}=\mathbb E \vert X \vert ^{p}<\infty.%
\]
This shows that sufficient {\bf Condition 2} of Lemma \ref{L1} holds, which by Proposition \ref{PR1} gives the result.

For the second part of $(i)$ observe that for any $K>0$
\[
\mathbb E\left(  \frac{\sum_{j=1}^{d}\vert X_{j}\vert ^{p}}{d}\right)
^{r/p}\geq \mathbb E\left(  \frac{\sum_{j=1}^{d}\vert X_{j}\vert
^{p}\mathbf 1\left\{  \vert X_{j}\vert \leq K\right\}  }{d}\right)  ^{r/p}.
\]
Observing that the right-hand side of the inequality converges to $(\mathbb E\vert
X\vert ^{p}\mathbf 1\{  \vert X\vert \leq K\})  ^{r/p}  $
as $d\to\infty$, we get for any $K>0$%
\[
\liminf_{d\to\infty}\mathbb E\left(  \frac{\sum_{j=1}^{d}\vert
X_{j}\vert ^{p}}{d}\right)  ^{r/p}\geq  \mathbb E\left(  \vert
X\vert ^{p}\mathbf 1\left\{  \vert X\vert \leq K\right\}  \right)  ^{r/p}.
\]
Since $K$ can be chosen arbitrarily large and we assume that $\mathbb E\vert
X\vert ^{p}=\infty$, we see that the conclusion holds.
\paragraph{Proof of $(ii)$} For the first part of $(ii)$, note that in this
case $r/p\geq1$, so $\varphi$ is convex. Moreover, note that
\begin{align*}
\mathbb E \left \vert \varphi \left ( \frac{\sum_{j=1}^d \vert X_j\vert^p}{d}\right)\right \vert & = \mathbb E \left ( \frac{\sum_{j=1}^d \vert X_j\vert^p}{d}\right )^{r/p} \\
& \leq d^{-1}{\mathbb E \vert X\vert ^r}\\
&\quad \mbox{(by Jensen's inequality)}\\
&<\infty.
\end{align*}
Thus sufficient {\bf Condition 1} of Lemma \ref{L1} holds,
which by Proposition \ref{PR1} leads to the result.

For the second part of $(ii)$, observe that if $\mathbb E\vert X\vert ^{r}%
=\infty$, then, for all $d\geq1$,
\[
\mathbb E\left(  \frac{\sum_{j=1}^{d} \vert X_{j}\vert ^{p}}{d}\right)
^{r/p}\geq d^{-r/p}\mathbb E\vert X\vert ^{r}=\infty.
\]
\qed
\end{proof}

Applying Corollary \ref{C1} with $p>0$ and $r=2$ yields the following important
result:
\begin{proposition}
Fix $p>0$ and assume that $0<\mathbb E \vert X\vert ^{m}<\infty$ for $m=\max(2,p)$. Then, as $d\to\infty$,
\begin{equation*}
\frac{\mathbb E\Vert \bX \Vert_p}{d^{1/p}}\to\mu_{p}^{1/p}
\label{d5}
\end{equation*}
and%
\begin{equation*}
\frac{\mathbb E\Vert \bX \Vert_p^{2}}{d^{2/p}}\to\mu_{p}^{2/p},
\label{d6}%
\end{equation*}
which implies
\[
\frac{\sqrt{\emph{Var}\Vert \bX\Vert _{p}}%
}{\mathbb E\Vert \bX\Vert _{p}}\to0,\text{ as } d \to \infty.%
\]
\end{proposition}

This result, when correctly stated, corresponds to Theorem 5 of \citet{FrWeVe07}.
It expresses the fact that the {\it relative standard deviation} converges towards zero when the dimension grows.
It is known in the computational learning literature as the
 $p$-norm concentration in high-dimensional spaces. It is noteworthy that, by Chebyshev's inequality, for all $\varepsilon>0$,
\begin{align}
\mathbb P \left \{\left |\frac{\Vert\bX\Vert_p}{\mathbb E \Vert\bX\Vert_p}-1\right|\geq \varepsilon \right\} &= \mathbb P \Big \{\left |\,\Vert\bX\Vert_p-\mathbb E \Vert\bX\Vert_p\,\right|\geq \varepsilon \mathbb E \Vert\bX\Vert_p\Big\} \nonumber\\
& \leq \frac{\text{Var}\Vert \bX\Vert _{p}}%
{\varepsilon^2\mathbb E^2\Vert \bX\Vert _{p}}\to0\text{, as }d\to
\infty.\label{TCI}
\end{align}
That is, ${\Vert\bX\Vert_p}/{\mathbb E \Vert\bX\Vert_p}\stackrel{\mathbb P}{\to} 1$ or, in other words, the sequence $\{\Vert\bX\Vert_p\}_{d \geq 1}$ is relatively stable \citep[][]{BoLuMa13}. This property guarantees that the random fluctuations of $\Vert\bX\Vert_p$ around its expectation are of negligible size when compared to the expectation, and therefore most information about the size of $\Vert\bX\Vert_p$ is given by $\mathbb E\Vert\bX\Vert_p$ as $d$ becomes large.
\subsection{Rates of convergence}
The asymptotic concentration statement of Corollary \ref{C1} can be made more precise by means of rates of convergence, at the price of stronger moment assumptions.  To reach this objective, we first need a general result to control the behavior of a function of an i.i.d.~empirical mean around its true value. Thus, assume that $\{Y_{j}\}_{j \geq 1}$ are i.i.d.~$Y$ with mean $\mu$ and
variance $\sigma^{2}$. As before, we define
$$\overline Y_d=d^{-1}\sum_{j=1}^d Y_j.$$
Let $\varphi$ be a real-valued function with
derivatives $\varphi^{\prime}$ and $\varphi^{\prime\prime}$. 
\citet{Kh04} provides sufficient conditions for
\begin{equation*}
\mathbb E\varphi (  \overline{Y}_d)  =\varphi (  \mu)
+\frac{\varphi^{\prime\prime} (  \mu)  \sigma^{2}}{2d}+\text{o}(
d^{-2}) \label{phi1}%
\end{equation*}
to hold. The following lemma, whose assumptions are less restrictive, can be used in place of Khan's result
\citeyearpar{Kh04}. For the sake of clarity, its proof is postponed to Section 4. 
\begin{lemma}
\label{LK}
Let $\{Y_{j}\}_{j \geq 1}$ be a sequence of i.i.d.~$Y$ random variables with mean $\mu$
and variance $\sigma^{2}$, and $\varphi$ be a real-valued function with
continuous derivatives $\varphi^{\prime}$ and $\varphi^{\prime\prime}$ in a
neighborhood of $\mu$. 
 Assume that for some $r>1$,
\begin{equation}
\mathbb E \vert Y\vert ^{r+1}<\infty
\label{rr}%
\end{equation}
and, with $1/s=1-1/r$,
\begin{equation}
\limsup_{d\to\infty}\mathbb E  \left\vert \varphi (  \overline
{Y}_{d})  \right\vert ^{s}  <\infty.\label{ss}%
\end{equation}
Then, as $d\to\infty$,
\begin{equation*}
\mathbb E\varphi(  \overline{Y}_d)  =\varphi (  \mu)
+\frac{\varphi^{\prime\prime}(  \mu) \sigma^2}{2d}+\emph{o}(
d^{-1}).\label{small}%
\end{equation*}
\end{lemma}

The consequences of Lemma \ref{LK} in terms of $p$-norm concentration
are summarized in the following proposition:
\begin{proposition}
\label{IFC}
Fix $p>0$ and assume that $0<\mathbb E \vert X\vert ^{m}<\infty$ for $m=\max(4,3p)$. Then, as $d\to\infty$,
\begin{equation*}
\mathbb E  \Vert \bX\Vert _{p}  =d^{1/p}\mu_{p}^{1/p}+\emph{O}(d^{1/p-1})
\label{d8}%
\end{equation*}
and%
\begin{equation*}
\emph{Var} \Vert \bX\Vert _{p}  =\frac{\mu_{p}^{2/p-2}%
\sigma_{p}^{2}}{d^{1-2/p}p^{2}}+\emph{o}(  d^{-1+2/p}),
\label{d7}%
\end{equation*}
which implies
\begin{equation*}
\frac{\sqrt{d\,\emph{Var}  \Vert \bX\Vert
_{p}  }}{\mathbb E  \Vert \bX\Vert _{p} }Ê\to \frac
{\sigma_{p}}{p\mu_{p}},\text{ as } d \to \infty.
\end{equation*}
\end{proposition}

Proposition \ref{IFC} shows that for a fixed large $d$, the {\it relative standard deviation} evolves
with $p$ as the ratio $\sigma_p/(p\mu_p)$. For instance, when the distribution of $X$ is uniform, 
$$\mu_p=\frac{1}{p+1}\quad \text{and} \quad \sigma_p=\frac{p}{p+1}\sqrt{\frac{1}{2p+1}}.$$
In that case, we conclude that
\begin{equation*}
\frac{\sqrt{d\,\mbox{Var}  \Vert \bX\Vert
_{p}  }}{\mathbb E \Vert \bX\Vert _{p}}Ê\to 
\sqrt{\frac{1}{2p+1}}.
\end{equation*}
Thus, in the uniform setting, the limiting {\it relative standard deviation} is a strictly decreasing function of $p$.
This observation is often interpreted by saying that $p$-norms are more concentrated for larger values of $p$. 
There are however distributions for which this is not the case. A counterexample is given by a balanced mixture of
two standard Gaussian random variables with mean $1$ and $-1$, respectively \citep[see][page 881]{FrWeVe07}. In that case,
it can be seen that the asymptotic {\it relative standard deviation} with $p\leq1$ is smaller than 
for values of $p\in[8,30]$, making fractional norms more concentrated.
\begin{proof}[Proposition \ref{IFC}]
Fix $p>0$ and introduce the functions on $\mathbb{R}$
\[
\varphi_{1}(  u)  =\vert u\vert ^{1/p}\quad \text{and}\quad \varphi_{2}( u)  =\vert u\vert ^{2/p}.
\]
Assume that $\mathbb E\vert X\vert
^{\max(  4,p)  }<\infty$. Applying Corollary \ref{C1} we get that, as $d\to \infty$,
\[
\mathbb E\left(  \frac{\sum_{j=1}^{d}\vert X_{j}\vert ^{p}}{d}\right)
^{2/p}\to\mu_{p}^{2/p}
\]
and%
\[
\mathbb E\left(  \frac{\sum_{j=1}^{d}\vert X_{j}\vert ^{p}}{d}\right)
^{4/p}\to\mu_{p}^{4/p}.
\]
This says that with $r=2$ and $s=2$, for $i=1,2$,
\[
\limsup_{d\to\infty}\mathbb E  \left \vert \varphi_{i}\left(  \frac{\sum_{j=1}%
^{d}\vert X_{j}\vert ^{p}}{d}\right)Ê\right \vert^s <\infty.%
\]
Now, let $Y=\vert X\vert ^{p}$. If we also assume that $\mathbb E\vert Y\vert ^{r+1}=\mathbb E\vert
Y\vert ^{3}=\mathbb E\vert X\vert ^{3p}<\infty$, we get by applying
Lemma \ref{LK} to $\varphi_{1}$ and $\varphi_{2}$ that for $i=1,2$%
\[
\mathbb E\varphi_{i}(  \overline{Y}_d)  =\varphi_{i}(  \mu_p)
+\frac{\varphi_{i}^{\prime\prime}(  \mu_p)  \sigma_p^{2}}{2d}+\text{o}(
d^{-1})  .
\]
Thus, whenever $\mathbb E\vert X\vert ^{m}<\infty$, where $m=\max(
4,3p)$, 
\begin{equation*}
\mathbb E \vert \overline{Y}_d\vert  ^{1/p}  =\mu_{p}^{1/p}+\frac
{1}{p}\left(  \frac{1-p}{p}\right)  \frac{\mu_{p}^{1/p-2}\sigma_{p}^{2}}%
{2d}+\text{o}(  d^{-1}) \label{p1}%
\end{equation*}
and
\begin{equation*}
\mathbb E \vert  \overline{Y}_d\vert  ^{2/p}=\mu_{p}^{2/p}+\frac
{1}{p}\left(  \frac{2-p}{p}\right)  \frac{\mu_p^{2/p-2}\sigma_{p}^{2}}%
{d}+\text{o}\left(  d^{-1}\right).
\label{p2}%
\end{equation*}
Therefore, we see that
\begin{align*}
\text{Var} \vert \overline{Y}_d\vert ^{1/p} &=\mathbb E \vert
\overline{Y}_d\vert  ^{2/p}  -\mathbb E ^{2} \vert \overline{Y}_d\vert
^{1/p}\nonumber\\
&=\frac{\mu_{p}^{2/p-2}\sigma_{p}^{2}}{dp^{2}}+\text{o}\left(  d^{-1}\right)
.
\label{p3}%
\end{align*}
The identity $\overline{Y}_d=d^{-1}\sum_{j=1}^d \vert X_j\vert ^p$ yields the desired results.
\qed
\end{proof}

We conclude the section with a corollary, which specifies inequality (\ref{TCI}).
\begin{corollary}
\label{fh}
Fix $p>0$.
\begin{enumerate}[$(i)$]
\item If $0<\mathbb E \vert X\vert ^{m}<\infty$ for $m=\max(4,3p)$, then, for all $\varepsilon >0$,
$$
\mathbb P \left \{ \left |\frac{\Vert\bX\Vert_p}{\mathbb E \Vert\bX\Vert_p}-1\right|\geq\varepsilon\right\} \leq \frac{\sigma_p^2}{\varepsilon ^2dp^2\mu_p^2}+\emph{o}(d^{-1}).
$$
\item If for some positive constant $C$, $0<|X|\leq C$ almost surely, then, for $p\geq 1$ and all $\varepsilon >0$,
$$
\mathbb P \left\{ \left |\frac{\Vert\bX\Vert_p}{\mathbb E \Vert\bX\Vert_p} -1\right|\geq \varepsilon \right\}\leq 2\exp \left (- \frac{\varepsilon ^2 d^{2/p-1}\mu_p^{2/p}}{2C^2}+\emph{o}(d^{2/p-1})\right).
$$
\end{enumerate}
\end{corollary}
\begin{proof} 
Statement $(i)$ is an immediate consequence of Proposition \ref{IFC} and  Chebyshev's inequality. Now, assume that $p \geq 1$, and let $A=[-C,C]$. For $\bx=(x_1, \hdots, x_d) \in \mathbb R^d$, let $g:A^d\to \mathbb R$ be defined by
$$g(\bx)=\Vert\bx\Vert_p=\left(\sum_{j=1}^d |x_j|^p\right)^{1/p}.$$
Clearly, for each $1 \leq j\leq d$,
\begin{align*}
&\sup_{\stackrel{(x_1, \hdots, x_d) \in A^d}{x'_j \in A}} \left | g(x_1, \hdots, x_d)-g(x_1, \hdots, x_{j-1}, x'_j, x_{j+1},\hdots, x_d)\right|\\
& \quad = \sup_{\bx \in A^d, x'_j \in A} \left |  \Vert\bx\Vert_p-\Vert\bx'\Vert_p\right|,
\end{align*}
where $\bx'$ is identical to $\bx$, except on the $j$-th coordinate where it takes the value $x'_j$. It follows, by Minkowski inequality (which is valid here since $p\geq 1$), that
\begin{align*}
&\sup_{\stackrel{(x_1, \hdots, x_d) \in A^d}{x'_j \in A}} \left | g(x_1, \hdots, x_d)-g(x_1, \hdots, x_{j-1}, x'_j, x_{j+1},\hdots, x_d)\right|\\
& \quad \leq \sup_{\stackrel{\bx \in A^d}{x'_j \in A}} \Vert\bx-\bx'\Vert_p\\
& \quad =\sup_{(x,x')\in A^2} |x-x'| \leq 2C.
\end{align*}
Consequently, using the bounded difference inequality \citep[][]{Mc89}, we obtain
\begin{align*}
\mathbb P \left\{ \left |\frac{\Vert\bX\Vert_p}{\mathbb E \Vert\bX\Vert_p} -1\right|\geq \varepsilon \right\}&=\mathbb P \bigg\{ \left |\,\Vert\bX\Vert_p-\mathbb E \Vert\bX\Vert_p\,\right| \geq \varepsilon \mathbb E \Vert\bX\Vert_p\bigg\}\\
& \leq 2 \exp \left (- \frac{2(\varepsilon \mathbb E\Vert\bX\Vert_p) ^2}{4dC^2}\right)\\
&  =2\exp \left (- \frac{\varepsilon ^2 d^{2/p-1}\mu_p^{2/p}}{2C^2}+\text{o}(d^{2/p-1})\right),
\end{align*}
where, in the last inequality, we used Proposition \ref{IFC}. This concludes the proof.
\qed
\end{proof}
\section{Minima and maxima}
Another important question arising in high-dimensional nearest neighbor
search analysis concerns the relative asymptotic behavior of the minimum and
maximum $p$-distances to the origin within a random sample. To be precise,
let $\mathbf{X}_{1},\hdots,\mathbf{X}_{n}$ be an i.i.d.~$\mathbf{X}$ sample,
where $\mathbf{X}=(X_{1},\hdots,X_{d})$ is as usual a $\mathbb{R}^{d}$%
-valued random vector with i.i.d.~$X$ components. We will be primarily
interested in this section in the asymptotic properties of the difference
(the {\it contrast}) $\max_{1\leq i\leq n}\Vert\mathbf{X}_{i}\Vert_{p}-\min_{1\leq
i\leq d}\Vert\mathbf{X}_{i}\Vert_{p}$. In other words, given a data set and
a fixed query point located---without loss of generality---at the origin, we
seek to analyze how much the distances to the farthest and nearest neighbors
differ.

Assume, to start with, that $n$ is fixed and only $d$ is allowed to grow. Then an
immediate application of the law of large numbers shows that, whenever $\mu
_{p}=\mathbb{E}|X|^{p}<\infty $, almost surely as $d\to \infty $, 
\begin{equation*}
d^{-1/p}\left( {\max_{1\leq i\leq n}\Vert \mathbf{X}_{i}\Vert
_{p}-\min_{1\leq i\leq n}\Vert \mathbf{X}_{i}\Vert _{p}}\right) \overset{%
\mathbb{P}}{\to}0.
\end{equation*}%
Moreover, if $0<\mu _{p}<\infty $, then 
\begin{equation*}
\frac{\max_{1\leq i\leq n}\Vert \mathbf{X}_{i}\Vert _{p}}{\min_{1\leq i\leq
n}\Vert \mathbf{X}_{i}\Vert _{p}}\overset{\mathbb{P}}{\to}1.
\end{equation*}%
The above ratio is sometimes called the {\it relative contrast} in the
computational learning literature. Thus, as $d$ becomes large, all
observations seem to be distributed at approximately the same $p$-distance
from the query point. The concept of nearest neighbor (measured by $p$%
-norms) in high dimension is therefore less clear than in small dimension, 
with resulting computational difficulties and algorithmic
inefficiencies.

These consistency results can be specified by means of asymptotic
distributions. Recall that if $Z_{1},\hdots,Z_{n}$ are i.i.d~standard normal
random variables, the sample range is defined to be 
\begin{equation*}
M_{n}=\max_{1\leq i\leq n}Z_{i}-\min_{1\leq i\leq n}Z_{i}.
\end{equation*}%
The asymptotic distribution of $M_{n}$ is well known (%
\citealp[see,
e.g.,][]{Da81}). Namely, for any $x$ one has%
\begin{align*}
& \lim_{n\to \infty }\mathbb P\left\{ \sqrt{2\log n}\left ( M_{n}-2\sqrt{%
2\log n}+\frac{\log \log n+\log 4\pi }{2\sqrt{2\log n}}\right) \leq
x\right\}  \\
& \quad =\int_{-\infty }^{\infty }\exp \left( -t-e^{-t}-e^{-(x-t)}\right) 
\text{\textrm{d}}t.
\end{align*}%
For future reference, we shall sketch the proof of this fact here. It is
well known that with 
\begin{equation}
a_{n}=\sqrt{2\log n}\quad \mbox{and}\quad b_{n}=\sqrt{2\log n}-\frac{1}{2}\frac{(\log
\log n+\log 4\pi )}{\sqrt{2\log n}}  \label{ab}
\end{equation}%
we have 
\begin{equation}
\left( a_{n}(\max_{1\leq i\leq n}Z_{i}-b_{n}),a_{n}(\min_{1\leq i\leq
n}Z_{i}+b_{n})\right) \overset{}{\to}(E,-E^{\prime }),  \label{EE}
\end{equation}%
where $E$ and $E^{\prime }$ are independent, $E$ $\overset{}{=}$ $E^{\prime }
$ and $\mathbb{P}\{E\leq x\}=\exp (-\exp (-x))$, $-\infty <x<\infty $. (\citealp[The
asymptotic independence of the maximum and minimum part can be inferred from
Theorem 4.2.8 of][]{Re89}, \citealp[and the asymptotic distribution part from
Example 2 on page 71 of][]{Re87}.) From (\ref{EE}) we get 
\begin{equation*}
a_{n}(\max_{1\leq i\leq n}Z_{i}-\min_{1\leq i\leq n}Z_{i})-2a_{n}b_{n}%
\overset{}{\overset{\mathcal{D}}{\to}}E+E^{\prime }.
\end{equation*}
Clearly, 
\begin{align*}
\mathbb{P}\{E+E^{\prime}\leq x\} & =\int_{-\infty}^{\infty}\exp\left(
-e^{-(x-t)}\right) \exp(-e^{-t})e^{-t}\text{\textrm{d}}t \\
& =\int_{-\infty}^{\infty}\exp\left( -t-e^{-t}-e^{-(x-t)}\right) \text{%
\textrm{d}}t.
\end{align*}
Our first result treats the case when $n$ is fixed and $d\to\infty$.

\begin{proposition}
\label{tcl} Fix $p>0$ and assume that $0<\mathbb E |X|^{p}<\infty $. Then, for fixed $%
n$, as $d\to\infty $, 
\begin{equation*}
d^{1/2-1/p}\left( {\max_{1\leq i\leq n}\Vert \mathbf{X}_{i}\Vert
_{p}-\min_{1\leq i\leq n}\Vert \mathbf{X}_{i}\Vert _{p}}\right) \overset{%
\mathcal{D}}{\to}\frac{\sigma _{p}\mu _{p}^{1/p-1}}{p}M_{n}.
\end{equation*}%
\end{proposition}

To our knowledge, this is the first statement of this type in the analysis
of high-dimensional nearest neighbor problems. In fact, most of the existing
results merely bound the asymptotic expectation of the (normalized)
difference and ratio between the max and the min, but with bounds which are
unfortunately not of the same order in $n$ as soon as $n\geq 3$ 
\citep[see,
e.g., Theorem 3 in][]{HiAgKe00}.

One of the consequences of Proposition \ref{tcl} is that, for fixed $n$, the
difference between the farthest and nearest neighbors does not necessarily
go to zero in probability as $d$ tends to infinity. Indeed, we see that the size of 
\begin{equation*}
{\max_{1\leq i\leq n}\Vert \mathbf{X}_{i}\Vert _{p}-\min_{1\leq i\leq
n}\Vert \mathbf{X}_{i}\Vert _{p}}
\end{equation*}
grows as $d^{1/p-1/2}$. For example, this difference increases with
dimensionality as $\sqrt{d}$ for the $L^{1}$ (Manhattan) metric and remains
stable in distribution for the $L^{2}$ (Euclidean) metric. It tends to
infinity  in probability for $p<2$ and to zero for $p>2$. This observation is in line with
the conclusions of \citet{HiAgKe00}, who argue that nearest neighbor search
in a high-dimensional space tends to be meaningless for norms with larger
exponents, since the maximum observed distance tends towards the minimum
one. It should be noted, however, that the variance of the limiting
distribution depends on the value of $p$.
\begin{remark}
Let $Z_{1},\hdots,Z_{n}$ be i.i.d~standard normal random variables, and let 
\begin{equation*}
R_{n}=\frac{\max_{1\leq i\leq n} Z_{i}}{\min_{1\leq i\leq n} Z_{i}}%
.
\end{equation*}%
Assuming $\mu _{p}>0$, one can prove, using the same technique, that 
\begin{equation*}
\frac{\max_{1\leq i\leq n}\Vert \mathbf{X}_{i}\Vert _{p}-d^{1/p}\mu _{p}}{%
\min_{1\leq i\leq n}\Vert \mathbf{X}_{i}\Vert _{p}-d^{1/p}\mu _{p}}\overset%
{}{\overset{\mathcal{D}}{\to}}R_{n}.
\end{equation*}%
\end{remark}
\begin{proof}[Proposition \ref{tcl}] Denote by $\mathbf{Z}_{n}$ a centered Gaussian random vector
in $\mathbb{R}^{n}$, with identity covariance matrix. By the central limit
theorem, as $d\to \infty $, 
\begin{equation*}
\sqrt{d}\left[ \left( \frac{\Vert \mathbf{X}_{1}\Vert _{p}^{p}}{d},\hdots,%
\frac{\Vert \mathbf{X}_{n}\Vert _{p}^{p}}{d}\right) -(\mu _{p},\hdots,\mu
_{p})\right] \overset{}{\overset{\mathcal{D}}{\to}}\sigma _{p}{}%
\mathbf{Z}_{n}.
\end{equation*}%
Applying the delta method with the mapping $f(x_{1},\hdots%
,x_{n})=(x_{1}^{1/p},\hdots,x_{n}^{1/p})$ (which is differentiable at $(\mu
_{p},\hdots,\mu _{p})$ since $\mu _{p}>0$), we obtain 
\begin{equation*}
\sqrt{d}\left[ \left( \frac{\Vert \mathbf{X}_{1}\Vert _{p}}{d^{1/p}},\hdots,%
\frac{\Vert \mathbf{X}_{n}\Vert _{p}}{d^{1/p}}\right) -(\mu _{p}^{1/p},\hdots%
,\mu _{p}^{1/p})\right] \overset{}{\overset{\mathcal{D}}{\to}}\frac{%
\sigma _{p}\mu _{p}^{1/p-1}}{p}\mathbf{Z}_{n}.
\end{equation*}%
Thus, by continuity of the maximum and minimum functions, 
\begin{equation*}
d^{1/2-1/p}\left( {\max_{1\leq i\leq n}\Vert \mathbf{X}_{i}\Vert
_{p}-\min_{1\leq i\leq n}\Vert \mathbf{X}_{i}\Vert _{p}}\right) \overset{%
\mathcal{D}}{\to}\frac{\sigma _{p}\mu _{p}^{1/p-1}}{p}M_{n}.
\end{equation*}%
\qed
\end{proof}

In the previous analysis, $n$ (the sample size) was fixed whereas $d$ (the
dimension) was allowed to grow to infinity. A natural question that arises
concerns the impact of letting $n$ be a function of $d$ such that $n$ tends
to infinity as $d\to \infty $ \citep[][]{Ma72}. Proposition \ref{Yu} below offers a
first answer.
\begin{proposition}
\label{Yu} Fix $p\geq 1$, and assume that $0<\mathbb{E}|X|^{3p}<\infty$ and $\sigma_p>0$. 
For any sequence of positive integers $%
\left\{ n( d) \right\} _{d\geq 1}$ converging to infinity and 
satisfying 
\begin{equation}
n( d) =\emph{o}\left( \frac{d^{1/5}}{\log^{6/5} d}\right) ,\text{
as }d\to \infty \text{,}  \label{nc}
\end{equation}%
we have%
\begin{equation*}
\frac{pa_{n( d) }d^{1/2-1/p}}{\mu _{p} ^{1/p-1}\sigma_p}%
\left( \max_{1\leq i\leq n( d) }\Vert {\mathbf{X}_{i}}\Vert
_{p}-\min_{1\leq i\leq n( d) }\Vert {\mathbf{X}_{i}}\Vert
_{p}\right) \overset{}{-2a_{n( d) }b_{n( d) }\overset{%
\mathcal{D}}{\to}}E+E^{\prime }\text{,}
\end{equation*}%
where $a_{n}$ and $b_n$ are as in (\ref{ab}), and $E$ and $E^{\prime }$ are as in (\ref{EE}).
\end{proposition}

\begin{proof} In the following, we let $\delta ( d) =1/\log d$.
For future use note that 
\begin{equation}
\delta ^{2}( d) \log n( d) \to 0\quad \mbox{and} \quad \frac{n^{5}(d)}{d\delta ^{6}( d)} \to
0,\ \text{as }d\to \infty \text{.}  \label{con}
\end{equation}%
In the proof we shall often suppress the dependence of $n$ and $\delta $ on $%
d$. For $1 \leq i \leq n$, we set 
\begin{equation*}
{\mathbf{X}_{i}}=(X_{1,i},\hdots,X_{d,i})\quad \mbox{and}\quad \Vert {%
\mathbf{X}_{i}}\Vert _{p}^{p}=\sum_{j=1}^{d}|X_{j,i}|^{p}.
\end{equation*}%
We see that for $n\geq 1,$%
\begin{align*}
& \left( \frac{\Vert {\mathbf{X}_{1}}\Vert _{p}^{p}-d\mu _{p}}{\sqrt{d}\sigma_p},%
\hdots,\frac{\Vert {\mathbf{X}_{n}}\Vert _{p}^{p}-d\mu _{p}}{\sqrt{d}\sigma_p}%
\right)  \\
& \quad =\left( \frac{\sum_{j=1}^{d}|X_{j,1}|^{p}-d\mu _{p}}{\sqrt{d}\sigma_p},\hdots%
,\frac{\sum_{j=1}^{d}|X_{j,n}|^{p}-d\mu _{p}}{\sqrt{d}\sigma_p}\right)  \\
& \quad :=(Y_{1},\hdots,Y_{n})=\mathbf{Y}_{n}\in \mathbb{R}^{n}.
\end{align*}%
As above, let $\mathbf{Z}{_{n}}=(Z_{1},\hdots,Z_{n})$ be a centered Gaussian
random vector in $\mathbb{R}^{n}$, with identity covariance matrix. Write, for $1 \leq j \leq d$,
\begin{equation*}
\bm{\xi}_{j}=\left( \frac{|X_{j,1}|^{p}-\mu _{p}}{\sqrt{d}\sigma_p},\hdots,\frac{%
|X_{j,n}|^{p}-\mu _{p}}{\sqrt{d}\sigma_p}\right) 
\end{equation*}%
and note that $\sum_{j=1}^{d}\bm{\xi}_{j}=\mathbf{Y}{_{n}}$. 
Set $\beta =\sum_{j=1}^{d}\Vert \bm{\xi}_{j}\Vert _{2}^{3}$. Then, by Jensen's
inequality, 
\begin{equation*}
\mathbb{E}\Vert \bm{\xi} _{j}\Vert _{2}^{3}=\mathbb{E}\left( \frac{%
\sum_{i=1}^{n}\left( |X_{j,i}|^{p}-\mu _{p}\right) ^{2}}{d\sigma^2_p}\right)
^{3/2}\leq \left( \frac{n}{d\sigma_p^2}\right) ^{3/2}\mathbb{E}\left\vert
\,|X|^{p}-\mu _{p}\,\right\vert ^{3}.
\end{equation*}%
This gives that for any $\delta >0$, possibly depending upon $n$, 
\begin{equation*}
B:=\beta n\delta ^{-3}\leq \frac{n^{5/2}}{\sqrt{d}\sigma_p^3}\mathbb{E}\left\vert
\,|X|^{p}-\mu _{p}\,\right\vert ^{3}\delta ^{-3}.
\end{equation*}%
Applying a result of \citet{Yu77} as formulated in Section 4 of Chapter 10
of \citet{Po01} we get, on a suitable probability space depending on $\delta
>0$, there exist random vectors $\mathbf{Y}_{n}$ and ${}\mathbf{Z}_{n}$
satisfying 
\begin{equation}
\mathbb{P}\bigg\{ \Vert \mathbf{Y}_{n}-\mathbf{Z}_{n}\Vert _{2}>3\delta
\Big\} \leq CB\left( 1+\frac{\left\vert \log (B)\right\vert }{{n}}%
\right), \label{yp}
\end{equation}%
where $C$ is a universal constant. Using the fact that 
\begin{equation*}
\left\vert \max_{1\leq i\leq n}x_{i}-\max_{1\leq i\leq n}y_{i}\right\vert
\leq \sqrt{\sum_{i=1}^{n}\left( x_{i}-y_{i}\right) ^{2}},
\end{equation*}%
we get, for all $\varepsilon >0$,
$$\mathbb{P}\left\{ a_n\vert \max_{1\leq i\leq n}Y_{i}-\max_{1\leq i\leq n}Z_{i}\vert
>\varepsilon \right\} \leq  \mathbb{P}\Big\{ \sqrt{2 \log n}\,\Vert \mathbf{Y}_{n}-\mathbf{Z}_{n}\Vert
_{2}>\varepsilon \Big\}.$$
Thus, for all $d$ large enough,
\begin{align*}
\mathbb{P}\left\{ a_n\vert \max_{1\leq i\leq n}Y_{i}-\max_{1\leq i\leq n}Z_{i}\vert
>\varepsilon \right\} & \leq  \mathbb{P}\Big\{ \sqrt{2 \log n}\,\Vert \mathbf{Y}_{n}-\mathbf{Z}_{n}\Vert
_{2}>3 \delta \sqrt{2 \log n}\Big\} \\
& \quad (\mbox{since $\delta\sqrt{\log n} \to 0$ as $d\to \infty$})\\
&= \mathbb{P}\Big\{ \Vert \mathbf{Y}_{n}-\mathbf{Z}_{n}\Vert
_{2}>3 \delta\Big\}.
\end{align*}
From (\ref{yp}), we deduce that for all $\varepsilon >0$ and all $d$ large enough,
\begin{equation*}
\mathbb{P}\left\{ a_n\vert \max_{1\leq i\leq n}Y_{i}-\max_{1\leq i\leq n}Z_{i}\vert
>\varepsilon \right\}  \leq CB\left( 1+\frac{\left\vert \log (B)\right\vert }{{n}}%
\right) .  
\label{BB}
\end{equation*}
But, by our choice of $\delta(d)$ and (\ref{con}), 
\begin{equation*}
B\left( 1+\frac{\left\vert \log (B)\right\vert }{{n}}\right)
\to 0,
\end{equation*}%
so that
$$a_n\vert \max_{1\leq i\leq n}Y_{i}-\max_{1\leq i\leq n}Z_{i}\vert =\text{o}_{\mathbb{P}}(1).$$
Similarly, one proves that
$$a_n\vert \min_{1\leq i\leq n}Y_{i}-\min_{1\leq i\leq n}Z_{i}\vert =\text{o}_{\mathbb{P}}(1).$$
Thus, by (\ref{EE}), we conclude that
\begin{equation}
\left( a_{n}(\max_{1\leq i\leq n}Y_{i}-b_{n}),a_{n}(\min_{1\leq i\leq
n}Y_{i}+b_{n})\right) \overset{}{\overset{\mathcal{D}}{\to}}%
(E,-E^{\prime }).  \label{E}
\end{equation}%

Next, we have 
\begin{align*}
& \left( a_{n}(\max_{1\leq i\leq n}Y_{i}-b_{n}),a_{n}(\min_{1\leq i\leq
n}Y_{i}+b_{n})\right)  \\
& =\left( a_{n}\left( \frac{\max_{1\leq i\leq n}\Vert {\mathbf{X}_{i}}\Vert
_{p}^{p}}{\sqrt{d}\sigma_p}-\frac{\sqrt{d}\mu _{p}}{\sigma_p}-b_{n}\right) ,\text{ }a_{n}\left( \frac{%
\min_{1\leq i\leq n}\Vert {\mathbf{X}_{i}}\Vert _{p}^{p}}{\sqrt{d}\sigma_p}-\frac{\sqrt{d}%
\mu _{p}}{\sigma_p}+b_{n}\right) \right)  \\
& =\left( a_{n}\left( \frac{\max_{1\leq i\leq n}\Vert {\mathbf{X}_{i}}\Vert
_{p}^{p}}{\sqrt{d}\sigma_p}-\beta _{n}\right) ,\text{ }a_{n}\left( \frac{\min_{1\leq
i\leq n}\Vert {\mathbf{X}_{i}}\Vert _{p}^{p}}{\sqrt{d}\sigma_p}-\beta _{n}^{\prime
}\right) \right) ,
\end{align*}%
where $\beta _{n}=\frac{\sqrt{d}\mu _{p}}{\sigma_p}+b_{n}$ and $\beta _{n}^{\prime }=\frac{\sqrt{d}%
\mu _{p}}{\sigma_p}-b_{n}$. Note that $a_{n}\to \infty $ and (\ref{E}) imply
that both 
\begin{equation}
\frac{\max_{1\leq i\leq n}\Vert {\mathbf{X}_{i}}\Vert _{p}^{p}}{\sqrt{d}\sigma_p}%
-\beta _{n}\overset{\mathbb{P}}{\to}0\quad \mbox{and}\quad \frac{%
\min_{1\leq i\leq n}\Vert {\mathbf{X}_{i}}\Vert _{p}^{p}}{\sqrt{d}\sigma_p}-\beta
_{n}^{\prime }\overset{\mathbb{P}}{\to}0.  \label{zero}
\end{equation}%
Observe also that by a two term Taylor expansion, for a suitable $\widetilde{%
\beta }_{n}$ between $\beta _{n}$ and $({\max_{1\leq i\leq n}\Vert {\mathbf{X%
}_{i}}\Vert _{p}^{p}})/({\sqrt{d}\sigma_p})$, 
\begin{align*}
& \frac{pa_{n}}{\beta _{n}^{1/p-1}}\left( \left( \frac{\max_{1\leq i\leq
n}\Vert {\mathbf{X}_{i}}\Vert _{p}^{p}}{\sqrt{d}\sigma_p}\right) ^{1/p}-\beta
_{n}^{1/p}\right)  \\
& \quad =a_{n}\left( \frac{\max_{1\leq i\leq n}\Vert {\mathbf{X}_{i}}\Vert
_{p}^{p}}{\sqrt{d}\sigma_p}-\beta _{n}\right)  \\
& \qquad +\frac{a_{n}}{\beta _{n}^{1/p-1}}\frac{1-p}{2p}\,\widetilde{\beta }%
_{n}^{1/p-2}\left( \frac{\max_{1\leq i\leq n}\Vert {\mathbf{X}_{i}}\Vert
_{p}^{p}}{\sqrt{d}\sigma_p}-\beta _{n}\right) ^{2}.
\end{align*}%
We obtain by (\ref{E}) and (\ref{zero}) that 
\begin{equation*}
a_{n}^{2}\left( \frac{\max_{1\leq i\leq n}\Vert {\mathbf{X}_{i}}\Vert
_{p}^{p}}{\sqrt{d}\sigma_p}-\beta _{n}\right) ^{2}\frac{\widetilde{\beta }%
_{n}^{1/p-2}}{a_{n}\beta _{n}^{1/p-1}}=\text{O}_{\mathbb{P}}\left( \frac{1}{%
a_{n}\beta _{n}}\right) =\text{o}_{\mathbb{P}}(1).
\end{equation*}%
Similarly,
\begin{align*}
& \frac{pa_{n}}{\left( \beta _{n}^{\prime }\right) ^{1/p-1}}\left( \left( 
\frac{\min_{1\leq i\leq n}\Vert {\mathbf{X}_{i}}\Vert _{p}^{p}}{\sqrt{d}\sigma_p}%
\right) ^{1/p}-\left( \beta _{n}^{\prime }\right) ^{1/p}\right)  \\
& \quad =a_{n}\left( \frac{\min_{1\leq i\leq n}\Vert {\mathbf{X}_{i}}\Vert
_{p}^{p}}{\sqrt{d}\sigma_p}-\beta _{n}^{\prime }\right) +\text{o}_{\mathbb{P}}(1).
\end{align*}%
Keeping in mind that $\beta _{n}$ $/\beta _{n}^{\prime }\to 1$, we
get%
\begin{align*}
& \frac{pa_{n}}{\beta _{n}^{1/p-1}}\left( \left( \frac{\max_{1\leq i\leq
n}\Vert {\mathbf{X}_{i}}\Vert _{p}^{p}}{\sqrt{d}\sigma_p}\right) ^{1/p}-\beta
_{n}^{1/p},\left( \frac{\min_{1\leq i\leq n}\Vert {\mathbf{X}_{i}}\Vert
_{p}^{p}}{\sqrt{d}\sigma_p}\right) ^{1/p}-\left( \beta _{n}^{\prime }\right)
^{1/p}\right)  \\
& \quad \overset{}{\overset{\mathcal{D}}{\to}}(E,-E^{\prime })
\end{align*}%
and hence%
\begin{equation*}
\frac{pa_{n}}{\beta _{n}^{1/p-1}}\left( \frac{\max_{1\leq i\leq n}\Vert {%
\mathbf{X}_{i}}\Vert _{p}}{( \sqrt{d}\sigma_p) ^{1/p}}-\frac{\min_{1\leq
i\leq n}\Vert {\mathbf{X}_{i}}\Vert _{p}}{( \sqrt{d}\sigma_p) ^{1/p}}%
-\beta _{n}^{1/p}+( \beta _{n}^{\prime }) ^{1/p}\right) \overset{}%
{\overset{\mathcal{D}}{\to}}E+E^{\prime }.
\end{equation*}%
Next notice that (\ref{nc}) implies that $b_n/ \sqrt d \to 0$, as $%
d\to \infty $. Thus, recalling
$$\frac{\beta_n}{\sqrt d u_p/\sigma_p}=1+\frac{b_n}{\sqrt d \mu_p/\sigma_p} \quad \mbox{and} \quad \frac{\beta'_n}{\sqrt d u_p/\sigma_p}=1-\frac{b_n}{\sqrt d \mu_p/\sigma_p},$$
we are led to
\begin{equation*}
\frac{pa_{n}}{\beta _{n}^{1/p-1}}\left( \beta _{n}^{1/p}-\left( \beta
_{n}^{\prime }\right) ^{1/p}\right) =2a_{n}b_{n}+\mbox{O}( a_{n}b_{n}^{2}\beta
_{n}^{-1}) =2a_{n}b_{n}+\mbox{o}( 1).
\end{equation*}%
Therefore we get%
\begin{equation*}
\frac{pa_{n(d) }d^{1/2-1/p}}{\mu _{p} ^{1/p-1}\sigma_p}%
\left( \max_{1\leq i\leq n( d) }\Vert {\mathbf{X}_{i}}\Vert
_{p}-\min_{1\leq i\leq n( d) }\Vert {\mathbf{X}_{i}}\Vert
_{p}\right) \overset{}{-2a_{n( d) }b_{n( d) }\overset{%
\mathcal{D}}{\to}}E+E^{\prime }.
\end{equation*}%
\qed
\end{proof}
\section{Proof of Lemma \protect\ref{LK}}
In the sequel, to lighten notation a bit, we set $\overline{Y}=\overline{Y}%
_{d}$. Choose any $\varepsilon >0$ and $\delta >0$ such that $\varphi $ has
continuous derivatives $\varphi ^{\prime }$ and $\varphi ^{\prime \prime }$
on $I_{\delta }=[\mu -\delta ,\mu +\delta ]$ and $|\varphi ^{\prime \prime
}(\mu )-\varphi ^{\prime \prime }(x)|\leq \varepsilon $ for all $x\in
I_{\delta }.$ We see that by Taylor's theorem that for $\overline{Y}\in
I_{\delta }$ 
\begin{equation}
\varphi (\overline{Y})=\varphi (\mu )+\varphi ^{\prime }(\mu )(\overline{Y}%
-\mu )+2^{-1}\varphi ^{\prime \prime }(\widetilde{\mu })(\overline{Y}-\mu
)^{2},  \label{T}
\end{equation}%
where $\widetilde{\mu }$ lies between $\overline{Y}$ and $\mu $. Clearly, 
\begin{align*}
& \left\vert \mathbb{E}\varphi (\overline{Y})-\varphi (\mu )-\frac{\sigma
^{2}\varphi ^{\prime \prime }(\mu )}{2d}\right\vert  \\
& \quad =\left\vert \mathbb{E}\left( \varphi (\overline{Y})-\left( \varphi
(\mu )+\varphi ^{\prime }(\mu )(\overline{Y}-\mu )+2^{-1}\varphi ^{\prime
\prime }(\mu )(\overline{Y}-\mu )^{2}\right) \right) \right\vert  \\
& \quad \leq \left\vert \mathbb{E}\left( \left\{ \varphi (\overline{Y}%
)-\left( \varphi (\mu )+\varphi ^{\prime }(\mu )(\overline{Y}-\mu
)+2^{-1}\varphi ^{\prime \prime }(\mu )(\overline{Y}-\mu )^{2}\right)
\right\} \mathbf{1}\{\overline{Y}\in I_{\delta }\}\right) \right\vert  \\
& \qquad +\mathbb{E}\left( \left\vert \varphi (\overline{Y})\right\vert 
\mathbf{1}\{\overline{Y}\notin I_{\delta }\}\right) +\mathbb{E}\left(
\left\vert P(\overline{Y})\right\vert \mathbf{1}\{\overline{Y}\notin
I_{\delta }\}\right) ,
\end{align*}%
where 
\begin{equation*}
P(y)=\varphi (\mu )+\varphi ^{\prime }(\mu )(y-\mu )+2^{-1}\varphi ^{\prime
\prime }(\mu )(y-\mu )^{2}.
\end{equation*}%
Now using (\ref{T}) and $|\varphi ^{\prime \prime }(\mu )-\varphi ^{\prime
\prime }(x)|\leq \varepsilon $ for all $x\in I_{\delta }$, we may write 
\begin{align*}
& \left\vert \mathbb{E}\left( \left\{ \varphi (\overline{Y})-\left( \varphi
(\mu )+\varphi ^{\prime }(\mu )(\overline{Y}-\mu )+2^{-1}\varphi ^{\prime
\prime }(\mu )(\overline{Y}-\mu )^{2}\right) \right\} \mathbf 1\{\overline{Y}\in
I_{\delta }\}\right) \right\vert   \\
& \quad \leq \frac{\varepsilon }{2}\,\mathbb{E}(\overline{Y}-\mu )^{2}=\frac{%
\varepsilon \sigma ^{2}}{2d}.  
\end{align*}%
Next, we shall bound 
\begin{equation*}
\mathbb{E}\left( \left\vert \varphi (\overline{Y})\right\vert \mathbf 1\{\overline{Y}%
\notin I_{\delta }\}\right) +\mathbb{E}\left( \left\vert P(\overline{Y}%
)\right\vert \mathbf 1 \{\overline{Y}\notin I_{\delta }\}\right) :=\Delta
_{d}^{(1)}+\Delta _{d}^{(2)}.
\end{equation*}%
Recall that we assume that for some $r>1$, condition (\ref{rr}) holds. In
this case, by Theorem 28 on page 286 of \citet{Pe75} applied with ``$r$'' replaced
by ``$r+1$'', for all $\delta >0$,
\begin{equation}
\mathbb{P}\left\{ {|\overline{Y}-\mu |}\geq \delta \right\} =\text{o}%
(d^{-r}).  \label{dr}
\end{equation}%
Then, by using H\"{o}lder's inequality and (\ref{ss}), we get%
\begin{equation*}
\Delta _{d}^{(1)}\leq  \left( \mathbb{E}\left\vert \varphi (\overline{Y%
})\right\vert ^{s} \right) ^{1/s}\left( \mathbb{P}\{\overline{Y}%
\notin I_{\delta }\}\right) ^{1/r}=\text{o}(d^{-1}).  
\end{equation*}%
We shall next bound $\Delta _{d}^{(2)}$. Obviously from (\ref{dr}) 
\begin{equation*}
\left\vert \varphi (\mu )\right\vert \mathbb{P}\{\overline{Y}\notin
I_{\delta }\}=\text{o}(d^{-1}).
\end{equation*}%
Furthermore, by Cauchy-Schwarz inequality and (\ref{dr}),
\begin{equation*}
\mathbb{E}\left\vert \varphi ^{\prime }(\mu )(\overline{Y}-\mu )\mathbf{1}\{%
\overline{Y}\notin I_{\delta }\}\right\vert \leq \left\vert \varphi ^{\prime
}(\mu )\right\vert \sigma d^{-1/2}\text{o}(d^{-r/2})=\text{o}(d^{-1}),
\end{equation*}%
and by H\"{o}lder's inequality with $p=(r+1)/2$ and 
\begin{equation*}
q^{-1}=1-p^{-1}=1-2/(r+1)=(r-1)/(r+1),
\end{equation*}%
we have%
\begin{align*}
& 2^{-1}\left\vert \varphi ^{\prime \prime }(\mu )\right\vert \mathbb{E}%
\left( (\overline{Y}-\mu )^{2}\mathbf{1}\{\overline{Y}\notin I_{\delta
}\}\right)  \\
& \quad \leq 2^{-1}\left\vert \varphi ^{\prime \prime }(\mu )\right\vert
\left( \mathbb{E}|\overline{Y}-\mu |^{r+1}\right) ^{2/(r+1)}\left( \mathbb{P}%
\{\overline{Y}\notin I_{\delta }\}\right) ^{1/q}.
\end{align*}%
Applying Rosenthal's inequality \citep[see equation (2.3) in][]{GiMaZa03} we
obtain%
\begin{align*}
\mathbb{E}|\overline{Y}-\mu |^{r+1}& =\mathbb{E}\left\vert
d^{-1}\sum_{i=1}^{d}(Y_{i}-\mu )\right\vert ^{r+1} \\
& \leq \left( \frac{15(r+1)}{\log (r+1)}\right) ^{r+1}\max \left( d^{-\left(
r+1\right) /2}\left( \mathbb{E}Y^{2}\right) ^{\left( r+1\right)
/2},d^{-r}\mathbb{E}|Y|^{r+1}\right) .
\end{align*}%
Thus 
\begin{equation*}
\left( \mathbb{E}|\overline{Y}-\mu |^{r+1}\right) ^{2/(r+1)}=\text{O}%
(d^{-1}),
\end{equation*}%
which when combined with (\ref{dr}) gives 
\begin{equation*}
2^{-1}\left\vert \varphi ^{\prime \prime }(\mu )\right\vert \left( \mathbb{E}%
|\overline{Y}-\mu |^{r+1}\right) ^{2/(r+1)}\left( \mathbb{P}\{\overline{Y}%
\notin I_{\delta }\}\right) ^{(r-1)/(r+1)}=\text{o}(d^{-1}).
\end{equation*}%
Thus 
\begin{equation*}
\Delta _{d}^{(2)}=\text{o}(d^{-1}).  
\end{equation*}%
Putting everything together, we conclude that for any $\varepsilon >0$%
\begin{equation*}
\limsup_{d\to \infty }d\left\vert \mathbb{E}\varphi (\overline{Y}%
_{d})-\varphi (\mu )-\frac{\sigma ^{2}\varphi ^{\prime \prime }(\mu )}{2d}%
\right\vert \leq \frac{\varepsilon \sigma ^{2}}{2}.
\end{equation*}%
Since $\varepsilon >0$ can be chosen arbitrarily small, this completes the
proof.

\bibliography{biblio-biau-mason}

\end{document}